\documentclass[a4paper,twoside,10pt]{article}
\usepackage[english]{babel}

\usepackage{amsmath,amsfonts,amssymb,amsthm}
\newtheorem{teo}{Theorem}
\newtheorem{lem}{Lemma}
\newtheorem{deff}{Definition}

 \title{{\bf  Theorem of Existence and Uniqueness of Solution  for Differential Equation of Fractional Order }}
\author{M.\,V.~Kukushkin \textsuperscript{1} \\ \\
  \small \textsuperscript{1} \textit{Russia, Geleznovodsk, kukushkinmv@rambler.ru} \\}
\date{}
\begin{document}

\maketitle

\begin{abstract}
In this paper we proved a theorems of existence and uniqueness of solutions of differential equation of second order with fractional derivative in the Kipriyanov sense  in lower terms. As a domain of definition of the functions we consider the n --- dimensional   Euclidean space. By a simple reduction of Kipriyanov operator  to the operator of fractional differentiation in the sense of Marchaud these results can be considered valid for the operator of fractional differentiation in the sense of Riemann-Liouville, because of  known fact coincidence of these operators on the  classes of functions representable by the fractional integral.
\end{abstract}
\begin{small}\textbf{Keywords:} Fractional derivative;   embedding theorems;  energetic space;
  energetic inequality; fractional integral;
strong accretive operator;  positive defined operator.
\end{small}

\section{Brief historical review}
  In 1960, the famous mathematician   Kipriyanov I.A. in his paper \cite{firstab_lit:1kipriyanov1960} focuses  to the properties of the
eponymous  operator was formulated the theorem of existence and uniqueness of solutions for partial differential equations  second order with  operator fractional differentiation  in the lower terms,
it is noteworthy that the proof of this theorem was not published.
  Mathematicians  Djrbashian M.M., Nakhushev A.M.  ones  of the first in their works researched the differential equation second order  with fractional derivatives in the lower terms.
In 1970 was published the    work of  Djrbashian M.M. \cite{firstab_lit:1Jrbashyan1970}, in which is probably the first time considered  the problem of    eigenvalues of the differential operator fractional order.
In 1977 was published the    work of Nakhushev A.M. \cite{firstab_lit:1Nakhushev1977}. The author was considering the differential operator   second order with fractional derivatives in the sense of Riemann-Liouville in   lower terms. In this work was proved subsequently acquiring  a great value  theorem establishes a relationship between the eigenvalues of    homogeneous differential equation of   second order  with fractional derivative in lower terms and the zeros of functions  Mittag-Leffler type. Research in this direction was continued by Aleroev T.S.,  in 1982  was published his work \cite{firstab_lit:1Aleroev1982} in which he establishes a relationship between the zeros of an entire function and eigenvalues of the  boundary value problems for differential equations second order with fractional derivatives in the lower terms.
  It should be noted the monograph  of  Pskhu A.V. \cite{firstab_lit:Pskhu2005}  focuses to the partial differential equations of fractional order,  which was published in 2005. Bangti Jean and William Randall in their paper \cite{firstab_lit:1Bangti Jin2012} 2012   considered the inverse problem to the Sturm-Liouville problem for differential operator   second order with fractional derivative in the lower terms.
  It remains to note that the theory of differential equations of fractional order is still relevant today.

\section{Introduction}
 Accepting  a  notation   \cite{firstab_lit:kipriyanov1960} we assume that $\Omega$ --- convex domain of $n$ --- dimensional Euclidean space $\mathbb{E}^{n}$, $P$ is a fixed point of the boundary $\partial\Omega,$
$Q(r,\vec{\mathbf{e}})$ is an arbitrary point of $\Omega;$ we denote by $\vec{\mathbf{e}}$ is a unit vector having the direction from $P$ to $Q,$ using $r$ is the Euclidean distance between points $P$ and $Q.$
We will consider classes of Lebesgue $L_{p}(\Omega),\;1\leq p<\infty $  complex valued functions. In polar coordinates summability $f$ on $\Omega$ of degree $p,$ means that
\begin{equation*}
\int\limits_{\Omega}|f(Q)|^{p}dQ=\int\limits_{\omega}d\chi\int\limits_{0}^{d(\vec{\mathbf{e}})}|f(Q)|^{p}r^{n-1}dr<\infty,
\end{equation*}
where $d\chi$ --- is the element of the solid angle
the surface of a unit sphere in $\mathbb{E}^{n}$ and $\omega$ ---   surface of this sphere,   $d:=d(\vec{\mathbf{e}})$ --- is the length of  segment of  ray going from point $P$ in the direction
$\vec{\mathbf{e}}$ within the domain $\Omega.$
 Notation  ${\rm Lip}\, \lambda,\;0<\lambda\leq1 $   means the set of functions satisfying the Holder-Lipschitz condition
$$
{\rm Lip}\, \lambda:=\left\{\rho(Q):\;|\rho(Q)-\rho(P)|\leq M r^{\lambda},\;P,Q\in \bar{\Omega}\right\}.
$$
The operator of fractional differentiation in the  sense of Kipriyanov   defined in \cite{firstab_lit:1kipriyanov1960}  by formal expression
\begin{equation*}
\mathfrak{D}^{\alpha}(Q)=\frac{\alpha}{\Gamma(1-\alpha)}\int\limits_{0}^{r} \frac{[f(Q)-f(P+\vec{\mathbf{e}}t)]}{(r - t)^{\alpha+1}} \left(\frac{t}{r} \right) ^{n-1} dt+
C^{(\alpha)}_{n} f(Q) r ^{ -\alpha},\; P\in\partial\Omega,
$$
$$
C^{(\alpha)}_{n} = (n-1)!/\Gamma(n-\alpha),
\end{equation*}
according to theorem 2   \cite{firstab_lit:1kipriyanov1960} acting as follows
\begin{equation}\label{1}
\mathfrak{D}^{\alpha}:\stackrel{0}{W_p ^l} (\Omega)\rightarrow L_{q}(\Omega),
 \;lp\leq n,\;0<\alpha<l- \frac{n}{p} +\frac{n}{q},\; p\leq q<\frac{np}{n-lp}.
\end{equation}
If in the condition  \eqref{1}  we have  the strict inequality $q>p,$  then    for sufficiently small $\delta>0$   the next inequality holds
\begin{equation}\label{2}
\|\mathfrak{D}^{\alpha}f\|_{L_{q}(\Omega)}\leq \frac{K}{\delta^{\nu}}\|f\|_{L_{p}(\Omega)}+\delta^{1-\nu}\|f\|_{L^{l}_{p}(\Omega)},
\end{equation}
where
\begin{equation*}
 \nu=\frac{n}{l}\left(\frac{1}{p}-\frac{1}{q} \right)+\frac{\alpha+\beta}{l}.
\end{equation*}
The constant  $K$ independents on $\delta,\;f$ and point $P\in\partial\Omega ;\;\beta$ --- an arbitrarily small fixed positive number.
Further we assume that $(0<\alpha<1).$
  Denote ${\rm diam}\,\Omega =\mathfrak{d};\;C,C_{i}={\rm const},\;i\in \mathbb{N}_{0}.$ We use for inner product of points $P=(P_{1},P_{2},...,P_{n}) $ and $Q=(Q_{1},Q_{2},...,Q_{n})$ which    belong to  $\mathbb{E}^{n}$ a contracted notations $P\cdot Q=P^{i}\overline{Q_{i}}=\sum^{n}_{i=1}P_{i}\overline{Q_{i}},$ denote $|P-Q|=r$ ---- Euclidean distance  between $P$ and $Q.$ As usually  denote  $D_{i}u$ --- the generalized derivative of function $u$ with respect to coordinate variable with index   $1\leq i\leq n$ and let $Du=(D_{1}u,D_{2}u,...,D_{n}u).$ Denote $\vec{ \mathrm{e }}_{k},\,1\leq k \leq n$ --- ort on $n$ --- dimensional Euclidean space, and define the difference attitude $ \triangle^{h}_{k}v=
 [v(Q+  \vec{\mathrm{e}}_{k}h)-v(Q)]/h .$
We will assume that all functions has a zero extension outside  of $\bar{\Omega}.$
Everywhere further,  if not stated otherwise we will use the notations of   \cite{firstab_lit:kipriyanov1960}, \cite{firstab_lit:1kipriyanov1960}.

We define the familie    of operators $ \psi^{-}_{\varepsilon },\;\varepsilon>0$ as follows: $  \mathrm{D} (\psi^{-}_{  \varepsilon })\subset L_{p}(\Omega).$
In the right-side case
\begin{equation*}
 (\psi^{-}_{  \varepsilon }f)(Q)=  \left\{ \begin{aligned}
 \int\limits_{r+\varepsilon }^{d }\frac{ f (P+\vec{\mathbf{e}}r)- f(P+\vec{\mathbf{e}}t)}{( t-r)^{\alpha +1}} dt,\;0\leq r\leq d -\varepsilon,\\
   \frac{ f(Q)}{\alpha} \left(\frac{1}{\varepsilon^{\alpha}}-\frac{1}{(d -r)^{\alpha} }    \right),\;\;\;d -\varepsilon <r \leq d .\\
\end{aligned}
 \right.
 \end{equation*}
Following \cite[p.181]{firstab_lit:samko1987}  we define a truncated fractional derivative similarly the derivative in the sense of Marchaud, in the right-side case
 \begin{equation*}
 ( \mathfrak{D }^{\alpha}_{d-,\varepsilon}f)(Q)=\frac{1}{\Gamma(1-\alpha)}f(Q)(d-r)^{-\alpha}+\frac{\alpha}{\Gamma(1-\alpha)}(\psi^{-}_{  \varepsilon }f)(Q).
 \end{equation*}
  Right-side fractional derivatives  accordingly  will be understood  as a limit  in the sense of norm  $L_{p}(\Omega),\,1\leq p<\infty$ of truncated fractional derivative
 \begin{equation*}
 \mathfrak{D }^{\alpha}_{d-}f=\lim\limits_{\stackrel{\varepsilon\rightarrow 0}{ (L_{p}) }} \mathfrak{D }^{\alpha}_{d-,\varepsilon} f .
\end{equation*}

Consider a uniformly elliptic operator with real-valued   coefficients and fractional derivative in the sense of Kipriyanov  in the lower terms,    defined by   expression
\begin{equation}\label{3}
 Lu:=-  D_{j} ( a^{ij} D_{i}u)  +p\, \mathfrak{D}^{ \alpha }u,\;\;  i,j=1,2,...,n\, ,
 \end{equation}
 \begin{equation}\label{4}
 \; \mathfrak{D}(L)=H^{2}(\Omega)\cap H^{1}_{0}(\Omega),
 \end{equation}
 \begin{equation}\label{5}
 a^{ij}(Q)\in C^{1}(\bar{\Omega})  ,\;a^{ij}\xi _{i}  \xi _{j}  \geq a_{0}  |\xi|^{2},\,a_{0}>0,\;p(Q)>0,\;p(Q)\in {\rm Lip\,\lambda},\, \lambda>\alpha.
\end{equation}
We will use  a special case of the Green's formula
\begin{equation}\label{6.1}
-\int\limits_{\Omega}v\,\overline{D_{j}(a^{ij}D_{i}u)}\, dQ=\int\limits_{\Omega}a^{ij}D_{j}v\, \overline{D_{i}u}\,  dQ\,,\;u\in H^{2}(\Omega),v\in H_{0}^{1}(\Omega) .
\end{equation}
In later we  will need a following lemma.
\begin{lem}\label{L1}
Let $u,v\in L_{2}(\Omega),\;  {\rm dist}\,({\rm supp}\,v,\,\partial\Omega)>2|h|,$ then we have a following formula
 \begin{equation}\label{16}
\int\limits_{\Omega}\triangle^{ h}_{k}v\,\overline{u}\,dQ =
-\int\limits_{\Omega}v\, \triangle^{ -h}_{k}\overline{u} \,dQ.
\end{equation}
\end{lem}
\begin{proof}
In assumptions of this lemma we have a following
$$
\int\limits_{\Omega}\triangle^{ h}_{k}v\,\overline{u}\,dQ=\frac{1}{h}\int\limits_{\Omega} \left[v(Q+e_{k}h)-v(Q)\right]\,\overline{u(Q)}\,dQ =
 $$
$$
= \frac{1}{h}\int\limits_{\omega}d\chi \int\limits_{0}^{r}  v(P'+\mathbf{\bar{e}}r ) \,\overline{u(P'+\mathbf{\bar{e}}r-e_{k}h)}\,r^{n-1}dr-\frac{1}{h}\int\limits_{\Omega}  v(Q) \,\overline{u(Q)}\,dQ=
$$
$$
=\frac{1}{h}\int\limits_{\Omega'}   v(Q' ) \,\overline{u(Q'-e_{k}h)}\,dQ'-\frac{1}{h}\int\limits_{\Omega}  v(Q) \,\overline{u(Q)}\,dQ,\;P'=P+e_{k}h,\;Q'=P'+\mathbf{\bar{e}}r,
$$
where $\Omega'$ shift of the domain $\Omega$ on the distance $h$ in the direction $e_{k}.$
Note that   in consequence of   condition on the set: ${\rm supp}\,u,$ we have: ${\rm supp}\,u_{1}\subset \Omega\cap \Omega',\;u_{1}(Q')=u(Q'-e_{k}h) .$ Hence, finely        we can rewrite the last relation as a following
 \begin{equation*}
\int\limits_{\Omega}\triangle^{ h}_{k}v\,\overline{u}\,dQ=\frac{1}{h}\int\limits_{\Omega}   v(Q ) \overline{\left[ u(Q-e_{k}h)  -u(Q) \right]}\,dQ=
-\int\limits_{\Omega}v\,\triangle^{ -h}_{k}\overline{u}\,dQ.
\end{equation*}
\end{proof}

The theorems of existence and uniqueness which will be proved in the next section  based on the results obtained in the papers \cite{firstab_lit:1kukushkin2017}, \cite{firstab_lit:2kukushkin2017}.
\section{  Main theorems }
Consider the boundary value problem \eqref{3},\eqref{4}.
The proved a strong accretive property for  operators of fractional dierentiation provides the opportunity by using Lax-Milgram theorem    to prove the theorem of existence and uniqueness of generalized solution for this problem.
\begin{deff}\label{D1}
We will call the element $z\in H^{1}_{0}(\Omega) $ as a generalized solution of the boundary value problem \eqref{3},\eqref{4} if the following integral identity holds
\begin{equation}\label{6}
 B(v,z)= (v,f)_{L_{2}(\Omega)}  ,\;\forall v\in H^{1}_{0}(\Omega),
 \end{equation}
 where
\begin{equation*}
 B (v,u)= \int\limits_{\Omega} \left[ a^{ij}D_{j}v \overline{D_{i}u}  +      (\mathfrak{D}^{\alpha}_{d-}p\,v)  \, \overline{u} \right]\,dQ ,\;u,v\in H^{1}_{0}(\Omega).
\end{equation*}
\end{deff}

\begin{teo}\label{T1}

 There is an unique   generalized solution  of the   boundary value problem   \eqref{3},\eqref{4}.
\end{teo}
\begin{proof}
We will Show that the form \eqref{6} satisfies the conditions of Lax-Milgram theorem, particulary    we will show that  the next inequalities holds
  \begin{equation}\label{7}
  |B (v,u)|\leq K_{1}\|v\|_{H^{1}_{0}}\|u\|_{H^{1}_{0}} ,\;\;{\rm Re}\,B (v,v)\geq  K_{2}\|v\|^{2}_{H^{1}_{0} },\;u,v\in H^{1}_{0}(\Omega),
 \end{equation}
where  $K_{1}>0,\;K_{2}>0$ are constants independents from real functions  $u,v.$\\
Let us prove the first inequality of \eqref{7}.
 Using the Cauchy-Schwarz inequality for a sum,  we have
\begin{equation*}
a^{ij}  D_{j}v \overline{D_{i}u}\leq a(Q) |Dv||Du| ,\;a(Q)=\left(\sum\limits_{i,j=1}^{n}
|a_{ij}(Q)|^{2} \right)^{1/2}.
\end{equation*}
Hence
\begin{equation}\label{8}
 \left| \int\limits_{\Omega} a^{ij} D_{j}v \overline{D_{i}u}\, dQ\right|\leq    P   \|v\|_{H^{1}_{0}(\Omega)}\|u\|_{H^{1}_{0}(\Omega)},\;P=\sup\limits_{Q\in \Omega}|a(Q)|.
\end{equation}
In consequence of lemma 1 \cite{firstab_lit:1kukushkin2017},  lemma 2 \cite{firstab_lit:1kukushkin2017},    we have
\begin{equation}\label{9}
 (\mathfrak{D} ^{\alpha}_{d-}p\, v ,u    )_{L_{2}(\Omega )} =  ( v,\mathfrak{D}^{ \alpha }u)_{L_{2}(\Omega,p)},\;u,v\in H^{1}_{0}(\Omega)  .
\end{equation}
 Applying  the inequality \eqref{2}, then  Jung's inequality    we get
\begin{equation*}
\left|( v,\mathfrak{D}^{ \alpha }u)_{L_{2}(\Omega,p)}\right|     \leq C_{0}
  \|v\|_{L_{2}(\Omega)}\|\mathfrak{D}^{\alpha}u\|_{L_{q}(\Omega)} \leq
C_{0}  \|v\|_{L_{2}(\Omega)}\left\{\frac{K}{\delta^{\nu}}\|u\|_{L_{2}(\Omega)}+\delta^{1-\nu}\|u\|_{L^{1}_{2}(\Omega)} \right\} \leq
$$
$$
\leq\frac{1}{\varepsilon} \|v\|^{2}_{L_{2}(\Omega)} + \varepsilon\left(\frac{ KC_{0} }{\sqrt{2}\delta^{\nu}}\right)^{2} \|u\|^{2}_{L_{2}(\Omega)}  +
 \frac{\varepsilon}{2}\left( C_{0} \delta^{1-\nu}\right)^{2} \|u\|^{2}_{L^{1}_{2}(\Omega)}  ,
\end{equation*}
\begin{equation*}
 2<q<\frac{2n}{2\alpha-2+n},\;C_{0} = ({\rm mess}\,\Omega)^{\frac{q-2}{q}}\sup\limits_{Q\in\Omega}p(Q).
 \end{equation*}
Applying the Friedrichs inequality, finely we have a following estimate
\begin{equation}\label{10}
| (\mathfrak{D} ^{\alpha}_{d-}p\, v ,u    )_{L_{2}(\Omega )} |\leq C\|v\|_{H^{1}_{0}}\|u\|_{H^{1}_{0}}.
 \end{equation}
Note that  from inequalities  \eqref{8},\eqref{10} follows  the first inequality of \eqref{7}. Using the inequalities (28)\,\cite{firstab_lit:1kukushkin2017}, (36)\,\cite{firstab_lit:1kukushkin2017}, we have
\begin{equation}\label{11}
{\rm Re}\,B(v,v)\geq a_{0}\|v\|^{2}_{L_{2}^{1}(\Omega)}+\lambda^{-2}\|v\|^{2}_{L_{2}(\Omega,p)}\geq a_{0}\|v\|^{2}_{L_{2}^{1}(\Omega)}+\lambda^{-2}p_{0}\|v\|^{2}_{L_{2}(\Omega)},\;p_{0}=\inf\limits_{Q\in \Omega} p(Q).
\end{equation}
It is obviously that
\begin{equation*}
a_{0}\|v\|^{2}_{L_{2}^{1}(\Omega)}+\lambda^{-2}p_{0}\|v\|^{2}_{L_{2}(\Omega)}\geq K_{2}\left(\|v\|^{2}_{L_{2}^{1}(\Omega)}+
\|v\|^{2}_{L_{2}(\Omega)} \right)=
\end{equation*}
\begin{equation}\label{13}
=K_{2}\left(\int\limits_{\Omega}\sum\limits_{i=1}^{n}|D_{i}v|^{2}dQ+\int\limits_{\Omega} | v|^{2}dQ\right)=K_{2}\|v\|^{2}_{H_{0}^{1}},\;K_{2}=\min\{a_{0},\lambda^{-2}p_{0}\}.
\end{equation}
Hence  the second inequality of \eqref{7} follows from the  inequalities \eqref{11}, \eqref{13}.

Since conditions of Lax-Milgram theorem holds, then for all bounded on $H^{1}_{0}(\Omega) $ functional  $F,$  exist unique element $z\in H^{1}_{0}(\Omega) $ such as
 \begin{equation}\label{12.1}
   B (v,z)= F(v)  ,\;\forall v\in H^{1}_{0}(\Omega).
 \end{equation}
 Consider the functional
 \begin{equation}\label{14}
F(v)=(v,f)_{L_{2}(\Omega)},\,f\in L_{2}(\Omega),\,v\in H^{1}_{0}(\Omega).
 \end{equation}
Applying the  Cauchy-Schwarz inequality, we get
\begin{equation*}
|F(v)|= |( v,f)_{L_{2}(\Omega)}|\leq \|f\|_{L_{2}(\Omega)}\|v\| _{H_{0}^{1}(\Omega)}.
\end{equation*}
Hence the functional \eqref{14} is bounded on $H^{1}_{0}(\Omega),$ then in accordance with \eqref{12.1} we have   equality
 \begin{equation}\label{15}
   B ( v,z)= ( v,f)_{L_{2}(\Omega)}  ,\;\forall v\in H^{1}_{0}(\Omega).
 \end{equation}
 Therefore in accordance with definition  \ref{D1}  element   $z$ is an unique generalized solution of the   boundary value problem \eqref{3},\eqref{4}.
\end{proof}

The theorem \ref{T1} allows to prove the theorem of existence and uniqueness of   solution of the boundary value problem   \eqref{3},\eqref{4}.

 \begin{teo}\label{T2}
 There is  an unique strong solution  of the boundary value  problem \eqref{3},\eqref{4}.
\end{teo}
\begin{proof}

In consequence of  theorem \ref{T1} exists unique element $z\in H^{1}_{0}(\Omega),$ so that equality  \eqref{15} is true.
Note that if the generalized  solution   of boundary value problem \eqref{3},\eqref{4} belongs to Sobolev space $H^{2}(\Omega),$ then
applying formulas \eqref{6.1},\eqref{9} we get
\begin{equation}\label{9.1}
 (  v,Lz )_{L_{2}(\Omega )} = B( v,z)=( v,f)_{L_{2}(\Omega )},\;\forall v\in C_{0}^{\infty}(\Omega),
\end{equation}
hence
\begin{equation}\label{9.2}
 (  v,Lz-f )_{L_{2}(\Omega )} = 0,\;\forall v\in C_{0}^{\infty}(\Omega).
\end{equation}
Since it is well known that there is no non-zero element in the Hilbert space which is orthogonal to the dense manifold, then $z$ is solution of the boundary value problem\eqref{3},\eqref{4}.

Let's prove that $z\in H^{2} (\Omega).$ Choose the function $v$ in \eqref{15} so that $\overline{({\rm supp} \,v)}     \subset \Omega,$   performing easy calculation, using equality \eqref{9}, we get
\begin{equation}\label{17}
 \int\limits_{\Omega} a^{ij}  D_{j} v  \overline{D_{i} z}\, dQ = \int\limits_{\Omega}   v \overline{q}\, dQ ,\; \forall v\in H^{1}_{0}(\Omega),\,   \overline{({\rm supp} \,v)}     \subset \Omega,
\end{equation}
where $q=(f-p\,  \mathfrak{D}^{ \alpha }z).$
For $2|h|<  {\rm dist}\,({\rm supp}\,v,\,\partial\Omega) $ change the function $v$ on it  difference attitude $\triangle^{-h} v=\triangle^{-h}_{k}v $ for some $1\leq k\leq n.$ Using \eqref{16},   \eqref{17} we have
\begin{equation*}
 \int\limits_{\Omega}  D_{j} v\,\overline{\triangle^{ h}\left( a^{ij}  D_{i} z\right)} dQ=-\int\limits_{\Omega}  (D_{j}\triangle^{-h} v)\overline{a^{ij}D_{i} z}\, dQ=-\int\limits_{\Omega} a^{ij} (D_{j}\triangle^{-h} v)\overline{D_{i} z}\, dQ= -\int\limits_{\Omega}(\triangle^{-h}v)\, \overline{q} \, dQ.
\end{equation*}
Using elementary calculation we get
$$
\triangle^{ h}\left( a^{ij}  D_{i} z\right)(Q)=a^{ij}(Q+h \,\vec{\mathrm{e}}_{k})  (D_{i} \triangle^{ h} z )(Q)+[\triangle^{ h} a^{ij}(Q) ]( D_{i} z)(Q),
$$
hence
\begin{equation*}
 \int\limits_{\Omega}D_{j} v\, \overline{a^{ij}(Q+h \,\vec{\mathrm{e}}_{k})  (D_{i}\triangle^{ h}  z)} \, dQ=      -\int\limits_{\Omega}    Dv \cdot g +(\triangle^{-h}v)\,\overline{q}  \, dQ,
\end{equation*}
where $g=(g_{1},g_{2},...,g_{n}),\;g_{j}=(\triangle^{ h} a^{ij})   D_{i} z.$
Note last relation, using the Cauchy Schwarz inequality, finiteness property of function $v,$ lemma 7.23 \cite[p.164]{firstab_lit:gilbarg} we have
\begin{equation}\label{17.1}
\left|\int\limits_{\Omega}a^{ij}(Q+h \,\vec{\mathrm{e}}_{k}) D_{j} v\, \overline{  (D_{i}\triangle^{ h}  z)}\, dQ \right| = \left|\int\limits_{\Omega} D_{j} v\, \overline{a^{ij}(Q+h \,\vec{\mathrm{e}}_{k})  (D_{i}\triangle^{ h}  z)}\, dQ \right| \leq
 $$
 $$
\leq\|Dv\|_{L_{2}(\Omega)} \|g\|_{L_{2}(\Omega)}+\|\triangle^{-h}v\|_{L_{2}(\Omega)}\|q\|_{L_{2}(\Omega)} \leq\|Dv\|_{L_{2}(\Omega)}\left(\|g\|_{L_{2}(\Omega)}+\|q\|_{L_{2}(\Omega)}\right)  .
\end{equation}
Applying  the Cauchy Schwarz inequality for finite sum and integrals, it is easy to see that
$$
\|g\|_{L_{2}(\Omega)}=\left(\int\limits_{\Omega}\sum\limits_{j=1}^{n}|(\triangle^{ h} a^{ij})   D_{i} z|^{2}dQ\right)^{1/2}\leq \left(\int\limits_{\Omega}|Dz|^{2}\sum\limits_{i,j=1}^{n}| \triangle^{ h} a^{ij} |^{2}dQ\right)^{1/2}\leq
$$
$$
\leq \sup\limits_{Q\in \Omega}
\left(\sum\limits_{i,j=1}^{n}\left| \triangle^{ h} a^{ij}(Q) \right|^{2}\right)^{1/2}\left(\int\limits_{\Omega}|Dz|^{2}dQ\right)^{1/2}\leq C_{1} \|z\|_{H^{1}(\Omega)}.
$$
Note that using \eqref{2}, we have
$$
\|q\|_{L_{2}(\Omega)}\leq \|f\|_{L_{2}(\Omega)}+   \|p\,\mathfrak{D}^{ \alpha }z\|_{L_{2}(\Omega)} \leq \|f\|_{L_{2}(\Omega)}+C_{2}
  \|z\|_{H^{1} (\Omega)}.
 $$
Using given above  from \eqref{17.1}, we get
\begin{equation}\label{17.2}
 \left|\int\limits_{\Omega}a^{ij}(Q+h \,\vec{\mathrm{e}}_{k}) D_{j} v\, \overline{  (D_{i}\triangle^{ h}  z)}\, dQ \right| \leq C\left(\|z\|_{H^{1} (\Omega)}+\|f\|_{L_{2}(\Omega)}   \right)\|D  v\|_{L_{2}(\Omega)}.
\end{equation}
Note that using condition  \eqref{5}, we can get a following estimate 
\begin{equation}\label{17.3}
\left| \int\limits_{\Omega}  a^{ij} \xi_{j}\overline{\xi_{i}}\, dQ\right| =
\left|\int\limits_{\Omega}  a^{ij}(  {\rm Re}\xi_{j}\,{\rm Re}\xi_{i}+  {\rm Im} \xi_{j} \,{\rm Im} \xi_{i} )\, dQ+  
\imath \int\limits_{\Omega}  a^{ij}(  {\rm Re}\xi_{i}\,{\rm Im}\xi_{j}-  {\rm Re} \xi_{j} \,{\rm Im} \xi_{i} )\, dQ\right| =
$$
$$
=\left\{\left(\int\limits_{\Omega}  a^{ij}(  {\rm Re}\xi_{j}\,{\rm Re}\xi_{i}+  {\rm Im} \xi_{j} \,{\rm Im} \xi_{i} )\, dQ\right)^{2}+
 \left( \int\limits_{\Omega}  a^{ij}(  {\rm Re}\xi_{i}\,{\rm Im}\xi_{j}-  {\rm Re} \xi_{j} \,{\rm Im} \xi_{i} )\, dQ\right)^{2}\right\}^{1/2}\geq
 $$
 $$
 \geq \int\limits_{\Omega}  a^{ij}(  {\rm Re}\xi_{j}\,{\rm Re}\xi_{i}+  {\rm Im} \xi_{j} \,{\rm Im} \xi_{i} )\, dQ\geq 
 k_{0} \int\limits_{\Omega}    
 \left | \xi \right|^{2}\, dQ.
\end{equation} 
Define the function $\chi,$ so that: $  {\rm dist}\,({\rm supp}\,\chi,\,\partial\Omega)>2|h|,$
\begin{equation*}
\chi(Q)=\left\{ \begin{aligned}
 1,\;\;\;\;\;\;\;Q\in {\rm supp  }\,\chi,\\
  0,\;Q\in  \bar{\Omega} \setminus {\rm supp  }\,\chi .\\
\end{aligned}
 \right.
 \end{equation*}
Suppose that $v=\chi\triangle^{h}z.$ Using relations \eqref{17.2}, \eqref{17.3}, we have two-sided estimate
\begin{equation}\label{18}
k_{0}\|\chi\triangle^{ h} D   z\|^{2}_{L_{2}(\Omega)}\leq\left| \int\limits_{\Omega}\chi  a^{ij}(Q+h \,\vec{\mathrm{e}}_{k}) \triangle^{ h}  D_{j} z \overline{\triangle^{ h}D_{i}  z}\, dQ\right|=
 \left| \int\limits_{\Omega} a^{ij}(Q+h \,\vec{\mathrm{e}}_{k})     D_{j} (\chi\triangle^{h}z)\, \overline{ (D_{i}\triangle^{ h}  z)}\, dQ\right|\leq
$$
$$
\leq
C\left(\|z\|_{H^{1} (\Omega)}+\|f\|_{L_{2}(\Omega)}   \right)\|\chi\triangle^{ h} D   z\|_{L_{2}(\Omega)}.
\end{equation}
Using the  Jung's inequality, for all positive $k$ we  get an estimate
\begin{equation*}
2\left(\|z\|_{H^{1} (\Omega)}+\|f\|_{L_{2}(\Omega)}   \right)\|\chi\triangle^{ h} D   z\|_{L_{2}(\Omega)}\leq
\frac{1}{k} \left(\|z\|_{H^{1} (\Omega)}+\|f\|_{L_{2}(\Omega)}   \right)^{2}+ k \|\chi\triangle^{ h} D   z\|^{2}_{L_{2}(\Omega)} .
\end{equation*}
Choosing $k<2k_{0}C^{-1},$ we can perform inequality \eqref{18} as follows
\begin{equation*}
 \|\chi\triangle^{ h} D   z\|^{2}_{L_{2}(\Omega)}\leq
C_{1}\left(\|z\|_{H^{1} (\Omega)}+\|f\|_{L_{2}(\Omega)}   \right)^{2} .
\end{equation*}
It implies that for all domain $\Omega' ,\,{\rm dist}(\Omega',\partial \Omega)>2|h|,$ we have
\begin{equation*}
\|\triangle^{ h}_{i}  D_{j}z\|_{L_{2}(\Omega ') } \leq C_{2}\left(\|z\|_{H^{1} (\Omega)}+\|f\|_{L_{2}(\Omega)}   \right),
\;  i,j=1,2,...,n.
\end{equation*}
In consequence of lemma 7.24 \cite[p.165]{firstab_lit:gilbarg}, we have that exist generalized derivative $D_{i}  D_{j}  z$ and satisfies to condition
\begin{equation*}
 \|D_{i}  D_{j}  z\|_{L_{2}(\Omega  ) } \leq C_{2}\left(\|z\|_{H^{1} (\Omega)}+\|f\|_{L_{2}(\Omega)}   \right),
\;  i,j=1,2,...,n.
\end{equation*}
Hence $z\in H^{2}(\Omega).$
\end{proof}

\section{Conclusions}
 
Using the   Lax-Milgram  theorem we have proved the existence of a generalized solution of the boundary value problem for  differential equation fractional order.  Was proved the identity of function is a generalized solution to a   Sobolev class functions   corresponding to the strong solution. Although this method is not new in the theory of partial differential equations, when evidence was used a new technique of fractional calculus theory. The result is proved for multidimensional operator which has a  reduction to various operators of fractional order, in mind what the idea of the proof is of interest.

\end{document}